\documentclass{amsart}
\usepackage{amsfonts}
\usepackage{amsmath,amssymb}
\usepackage{amsthm}
\usepackage{amscd}
\usepackage{graphics}
\usepackage{graphicx}

\theoremstyle{remark}{
\newtheorem{Def}{{\rm Definition}}
\newtheorem{Ex}{{\rm Example}}
\newtheorem{Rem}{{\rm Remark}}
\newtheorem{Prob}{{\rm Problem}}

}
\theoremstyle{plain}
{
\newtheorem{Cor}{Corollary}
\newtheorem{Prop}{Proposition}
\newtheorem{Thm}{Theorem}
\newtheorem{MainThm}{Main Theorem}

}

\begin{document}
\title[Smooth functions associated to similar graphs of functions]{Reeb spaces of smooth functions associated to globally similar graphs of smooth functions}
\author{Naoki kitazawa}

\address{Osaka Central Advanced Mathematical Institute (OCAMI) \\
3-3-138 Sugimoto, Sumiyoshi-ku Osaka 558-8585
TEL: +81-6-6605-3103
}
\email{naokikitazawa.formath@gmail.com}
\urladdr{https://naokikitazawa.github.io/NaokiKitazawa.html}
\maketitle
\begin{abstract}

Previously, we have investigated a natural smooth map onto the region surrounded by the graphs of two smooth real-valued functions in the plane converging to a same value or diverges to $+\infty$ or $-\infty$ simultaneously, at each infinity, and topological properties and combinatorial ones of its composition with the canonical projection. Here, we consider smooth functions with {\t congruent} or {\it globally similar} graphs instead.

Here, the {\it Reeb space} of a smooth function on a manifold with no boundary is fundamental and important. This is the naturally topologized quotient space of the manifold, consisting of all connected components ({\it contours}) of the function and is a graph under a certain nice situation. Studies related to the present study were started due to interest of the author in theory of Reeb spaces of so-called {\it non-proper} functions. For {\it proper} functions, in 2020s, related studies have developed mainly due to Gelbukh and Saeki.

\end{abstract}
\section{Introduction.}
\label{sec:1}

The region $D_{c_1,c_2}$ surrounded by the graphs $\{(c_i(x),x) \mid x \in \mathbb{R}\}$ of nice smooth real-valued functions $c_i:\mathbb{R} \rightarrow \mathbb{R}$ ($i=1,2$) in the plane ${\mathbb{R}}^2$ is of fundamental mathematical objects.

First, the author has been interested in natural smooth maps onto such regions, which are locally so-called {\it special generic} maps. The class of {\it special generic} maps is the class containing the canonical projection of the {\it $m$-dimensional unit sphere} $S^m:=\{x =(x_1,\cdots, x_{m+1}) \in {\mathbb{R}}^{m+1} \mid {\Sigma}_{j=1}^{m+1} {x_j}^2=1\}$ onto the {\it $n$-dimensional unit disk} $D^n:=\{x=(x_1,\cdots x_n) \in {\mathbb{R}}^n \mid {\Sigma}_{j=1}^{n} {x_j}^2 \leq 1\}$, where ${\mathbb{R}}^k$ denotes the {\it $k$-dimensional Euclidean space} with $\mathbb{R}:={\mathbb{R}}^1$, and with $m$ and $n$ here being positive integers satisfying $m \geq n \geq 1$. For special generic maps, consult \cite{burletderham, furuyaporto, saeki1} and for our construction, \cite{kitazawa3}, where we present related arguments later, in self-contained ways. Second, we are mainly interested in the function obtained as the composition of the map
$f_{c_1,c_2}:X_{m,c_1,c_2} \rightarrow {\mathbb{R}^2}$ from an $m$-dimensional smooth submanifold in ${\mathbb{R}}^{m+1}$ onto the closure $\overline{D_{c_1,c_2}}$ of the open region $D_{c_1,c_2} \subset {\mathbb{R}}^2$ taken in ${\mathbb{R}}^2$.
 with ${\pi}_{2,1}:{\mathbb{R}}^2 \rightarrow \mathbb{R}$, where ${\pi}_{k_1,k_2}:{\mathbb{R}}^{k_1} \rightarrow {\mathbb{R}}^{k_2}$ denotes the projection ${\pi}_{k_1,k_2}(x):=x_1$ with $x:=(x_1,x_2) \in {\mathbb{R}}^{k_2} \times {\mathbb{R}}^{k_1-k_2}={\mathbb{R}}^{k_1}$ with $k_1>k_2 \geq 1$ be integers. Here $f_{c_1,c_2}$ is also constructed as the restriction of ${\pi}_{m+1,2}$ to $X_{m,c_1,c_2}$.  We are also interested in the Reeb space $R_{{\pi}_{2,1} \circ f_{c_1,c_2}}$ of the resulting function ${\pi}_{2,1} \circ f_{c_1,c_2}$. This is, roughly, the space of all connected components ({\it contours}) of {\it level sets} of the function.

Before explaining Reeb spaces rigorously, we introduce notation on differentiable manifolds and maps. The tangent vector space of a differentiable manifold $X$ at $p \in X$ is denoted by $T_p X$, which is a real vector space of dimension same as that of $X$.
For a differentiable map $f:X_1 \rightarrow X_2$ between differentiable manifolds, a {\it singular} point of it is a point $p \in X$ where the rank of the differential ${df}_p:T_p X_1 \rightarrow T_{f(p)} X_2$, which is of course a linear map, drops. We use $S(f)$ for the set of all singular points of $f$ and this is the {\it singular set} of $f$. 
The image $f(S(f))$ is the {\it singular value set} of $f$, where $f(p)$ ($p \in S(f)$) is a {\it singular value} of $f$.
In the case $X_2=\mathbb{R}, S^1$, or generally the case where $X_2$ is a $1$-dimensional manifold with no boundary, we use "{\it critical}" instead of "singular". We explain the {\it Reeb space} $R_c$ of a (continuous) real-valued function $c:X \rightarrow \mathbb{R}$ on a topological space $X$. The set $c^{-1}(q)$ is a {\it level set} of $c$ and a connected component of this is a {\it contour} of $c$. In the case of a differentiable function $c$, we add "{\it regular}" ("{\it critical}") before "contour" and "level set" if it contains no critical point (resp. some critical points) of $c$. We can induce the equivalence relation ${\sim}_c$ on $X$ as follows. We have $x_1 {\sim}_c x_2$ if and only if $x_1$ and $x_2$ are points of a same contour of $c$. We can have the quotient map $q_c:X \rightarrow R_c:=X/{{\sim}_c}$ and the unique continuous function $\bar{c}:R_c \rightarrow \mathbb{R}$ with $c=\bar{c} \circ q_c$. 

\begin{Prob}
\label{prob:1}
Does $R_c$ have the structure of a graph? 
\end{Prob}

Hereafter, a {\it graph} means a $1$-dimensional connected, locally compact and locally finite CW complex whose $1$-cell is called an {\it edge}, whose $0$-cell is called a {\it vertex}, and the closure of whose edge is homeomorphic to $D^1$. A {\it graph with ends} is defined by weakening the last condition on the closure of an edge, as a manifold homeomorphic to $D^1$ or $\{t>0\} \subset \mathbb{R}$. The {\it index} of a vertex is the number of edges incident to it. For (CW) complexes, see \cite{hatcher} for example.

$R_c$ is "1-dimensional", according to \cite{gelbukh1, gelbukh2, saeki2, saeki3}. This does not contradict our intuition.
 In the case of a so-called {\it Morse}({\it -Bott}) function $c$ on a compact and connected manifold, see \cite{izar, martinezalfaromezasarmientooliveira}. More specifically, in the case where $X$ is a smooth manifold with no boundary with $c:X \rightarrow \mathbb{R}$ being proper and $c(S(c))$ being discrete, $R_c$ is a graph (with ends) whose vertex set consists of all points for critical contours of $c$ and the {\it Reeb graph} of $c$. This has the structure of a digraph (with ends), by the following rule: an edge $e$ incident to a vertex $v_{e}$ is an oriented edge departing from (resp. entering) $v_{e}$ if the restriction of $\bar{c}$ to the closure of $e$ in the graph has the minimum (resp. maximum) at $v_{e}$. This is the {\it Reeb digraph} of $c$.

\begin{Prob}
\label{prob:2}
Can we construct a (nice) smooth function whose
Reeb space is the Reeb graph, isomorphic to a given graph (resp. with ends)?
\end{Prob}

This is one of fundamental, natural and incredibly new problems related to our studies. Sharko has launched this in 2006 (\cite{sharko}). This is a study on reconstruction of smooth functions with prescribed Reeb graphs on suitable closed surfaces whose critical points $p$ are of the form $c(z)={\Re} z^{l}+c(p) \in \mathbb{R}$ ($z \in \mathbb{C}$) or critical points of Morse functions. Here, $\mathbb{C}$ is the space of all complex numbers, $l>1$ is an integer, and ${\Re} s$ denotes the real part of $s \in \mathbb{C}$. This is extended to arbitrary finite graphs in \cite{masumotosaeki} and functions whose critical points may not be isolated are studied. In \cite{michalak}, a certain class of finite graphs and reconstructing Morse functions on closed manifolds with prescribed Reeb graphs and regular contours being spheres are discussed. Contribution to this by the author is, \cite{kitazawa1, kitazawa4}, for example. The author respects not only graphs, but also shapes of level sets, first. Related to this, in \cite{kitazawa2}, the author considers smooth functions on non-compact manifolds and {\it non-proper} functions, first. 
Remember that a {\it proper} map between topological spaces is a map the preimage of each compact set by which is also compact. The paper \cite{kitazawa3} is a pioneering study on reconstruction of functions represented by equations with real polynomials and with prescribed Reeb graphs. Motivated by this success, the author has been contributing to real algebraic construction: we do not assume related knowledge, and we omit precise exposition on other related papers and preprints, except \cite{kitazawa5, kitazawa6}. Related arguments appear in our main ingredients and we discuss them in self-contained ways (Theorem \ref{thm:1}).

Recently the author has started to consider infinite graphs (with ends) or spaces non-homeomorphic to graphs as Reeb spaces, non-proper function cases, and obtained several explicit situations. Related examples are presented as counterexamples in \cite{gelbukh1, gelbukh2, saeki2, saeki3}. The paper \cite{kitazawa7} is related to this. Different from these cases, the author has constructed them as, fully or densely, real algebraic or real analytic in \cite{kitazawa7, kitazawa8}. Arguments in the \cite{kitazawa3} and the preprints \cite{kitazawa5, kitazawa6} play important roles. This paper is regarded as another case of \cite{kitazawa9}. 

We go back to the beginning of this section. In \cite{kitazawa9}, at each infinity, $c_1$ and $c_2$ converge to a same value or both diverges to the same infinity $+\infty$ or $-\infty$. We study the case the graphs of $c_1$ and $c_2$ are {\it congruent} in ${\mathbb{R}}^2$ or more generally, "{\it globally similar}". We present some new result related to the previous preprint \cite{kitazawa9}. We argue in self-contained ways as we can. We discuss Problems \ref{prob:1} and \ref{prob:2}, mainly, Problem \ref{prob:1}, for the map $f_{c_1,c_2}$, the function ${\pi}_{2,1} \circ f_{c_1,c_2}$, and the Reeb space $R_{{\pi}_{2,1} \circ f_{c_1,c_2}}$. This is also a main topic of \cite{kitazawa9}.
Hereafter, for smooth functions $c_i$, we concentrate on smooth functions whose critical sets are disjoint unions of finitely many or countably many sets each of which is a one-point set, a copy of $D^1$, $\{t \geq 0\}$, or $\mathbb{R}$. 
For example, consider real analytic functions. We call such a function a {\it real-valued function on the real-line with tame singularities}, or an {\it $ \mathbb{R}$-$\mathbb{R}$ TS} function. For such a function $c_i$ with $S(c_i)$ being non-empty, we have a representation $S(c_i)={\sqcup}_{j=1}^l {I_{i,j}}$ with $l$ being a positive integer, or $S(c_i)={\sqcup}_{j=1}^{\infty} {I_{i,j}}$ or $S(c_i)={\sqcup}_{j=-\infty}^{\infty} {I_{i,j}}$ with $I_{i,j}$ being a closed and connected non-empty set as above. The set $c_i(I_{i,j})$ is a one-point set and let $c_{i,j} \in \mathbb{R}$ denote the unique element there. We present some of our main result.
 \setcounter{MainThm}{-1}
\begin{MainThm}
\label{mthm:0}
\begin{enumerate}
\item \label{mthm:0.1} {\rm (}Main Theorem \ref{mthm:1}{\rm )} For two $ \mathbb{R}$-$\mathbb{R}$ TS functions $c_1$ and $c_2$ with $c_1(S(c_1)) \bigcup c_2(S(c_2))$ being discrete and closed in $\mathbb{R}$, the Reeb space $R_{{\pi}_{2,1} \circ f_{c_1,c_2}}$ is regarded to be a $1$-dimensional CW complex.
\item \label{mthm:0.2} {\rm (}Main Theorem \ref{mthm:3}{\rm )} Let $c_{0}$ be an $ \mathbb{R}$-$\mathbb{R}$ TS function $c_{0}$ which is bounded and whose 1st derivative ${c_{0}}^{\prime}(x)$ satisfies $-\frac{1}{a_{c,{\rm M}}} \leq {c_{0}}^{\prime}(x) \leq a_{c,{\rm m}}$ for every $x \in \mathbb{R}$, for some two real numbers $a_{c,{\rm m}}, a_{c,{\rm M}}>0$ with $a_{c,{\rm m}}<a_{c,{\rm M}}$. Let $0<a_c<a_{c,{\rm M}}$ and we also assume 
 that the level set ${{c_{0}}^{\prime}}^{-1}(a_{c})$ is a non-empty set ${\sqcup}_{j=1}^l {I_{0,j}}$ with $l$ being a positive integer, ${\sqcup}_{j=1}^{\infty} {I_{0,j}}$, or ${\sqcup}_{j=-\infty}^{\infty} {I_{0,j}}$, with $I_{0,j}$ being a closed and connected non-empty set as above, and that the union of tangent lines of the graph $\{(c_0(x),x) \mid x \in \mathbb{R}\}$ at all points of the image ${c_0}({{c_{0}}^{\prime}}^{-1}(a_{c}))$ is the disjoint union of at most countably many tangent lines and also a closed set of ${\mathbb{R}}^2$. By a suitable rotation, we have another function $c_1$ and by $c_2:=c_1+a$ {\rm (}$a \in \mathbb{R}${\rm )}, we have a situation of \cite[Theorem 4]{kitazawa9}.
\end{enumerate}
\end{MainThm}
Main Theorem \ref{mthm:0} is revisited in revised ways, and we present additional cases, prove our new result, and present future problems, in the next section.

\section{Our main result.}
\subsection{Main Theorems.}

Theorem \ref{thm:1} is essential in defining the important map $f_{c_1,c_2}$ onto the closure $\overline{D_{c_1,c_2}}$ of the region $D_{c_1,c_2}$ surrounded by the graph, where the closure is considered in ${\mathbb{R}}^2$. This is essentially due to \cite{kitazawa3}. As we do in the preprint \cite{kitazawa9}, we review again. Some introductory propositions and theorems are same as those in \cite{kitazawa9}.
\begin{Thm}
	\label{thm:1}
	Let $m \geq 2$ be an integer.
	Let $c_i:\mathbb{R} \rightarrow \mathbb{R}$ {\rm (}$i=1,2${\rm )} be smooth functions satisfying $c_1(x)<c_2(x)$ for any $x \in \mathbb{R}$.
	Let $D_{c_1,c_2}:=\{(x_1,x_2)\mid c_1(x_2)<x_1<c_2(x_2)\}$ and we use $\overline{D_{c_1,c_2}}$ for its closure in ${\mathbb{R}}^2$.
	Then $X_{m,c_1,c_2}:=\{(x_1,x_2,{(y_j)}_{j=1}^{m-1}) \in \overline{D_{c_1,c_2}} \times {\mathbb{R}}^{m-1} \subset {\mathbb{R}}^{m+1} \mid (x_1-c_1(x_2))(c_2(x_2)-x_1)-{\Sigma}_{j=1}^{m-1} {y_j}^2=0\}$ is an $m$-dimensional smooth submanifold with no boundary in ${\mathbb{R}}^{m+1}$, the restriction $f_{c_1,c_2}:={\pi}_{m+1,2} {\mid}_{X_{m,c_1,c_2}}$ is a smooth surjection onto $\overline{D_{c_1,c_2}}$, and $S({\pi}_{m+1,2} {\mid}_{X_{m,c_1,c_2}})=\{(x_1,x_2,{(y_j)}_{j=1}^{m-1}) \mid (x_1,x_2) \in \overline{D_{c_1,c_2}}-D_{c_1,c_2}\}=\{(x_1,x_2,{(0)}_{j=1}^{m-1}) \mid (x_1,x_2) \in \overline{D_{c_1,c_2}}-D_{c_1,c_2}\}$.
	\end{Thm}
\begin{proof}
	We prove this in a self-contained way. A main ingredient is implicit function theorem and by this we show that the subset of ${\mathbb{R}}^{m+1}$ is a smooth manifold with no boundary. 

First, this subset $X_{m,c_1,c_2}$ is the zero set of the smooth function $(x_1-c_1(x_2))(c_2(x_2)-x_1)-{\Sigma}_{j=1}^{m-1} {y_j}^2$, by the definition.
	
	In the case $(x_1,x_2) \in D_{c_1,c_2}$ in this zero set, at the point, the value of the partial derivative of $(x_1-c_1(x_2))(c_2(x_2)-x_1)-{\Sigma}_{j=1}^{m-1} {y_j}^2$ by some $y_j$ is not $0$.
	
	In the case $(x_1,x_2) \in \overline{D_{c_1,c_2}}-D_{c_1,c_2}$ in this zero set, at the point, the partial derivative of $(x_1-c_1(x_2))(c_2(x_2)-x_1)-{\Sigma}_{j=1}^{m-1} {y_j}^2$ by $x_1$ is $c_1(x_2)-x_1$ or $c_2(x_2)-x_1$, which is not $0$.

This completes the proof.
	\end{proof}
Hereafter, we abuse the notation from Theorem \ref{thm:1}.
Theorem \ref{thm:2} (\ref{thm:2.1}) is also presented in \cite{kitazawa9}. We can easily check Theorem \ref{thm:2} (\ref{thm:2.2}) as an easy exercise on theory of {\it Morse} functions. 
We only explain that a smooth-real function is {\it Morse} if the {\it Hessian} at each critical point of the function, the symmetric matrix defined by considering the values of 2nd partial derivatives of the function, is non-degenerate, for some local coordinate in the manifold of the domain: this property does not depend on the local coordinate of the smooth manifold. Most of explicit functions in the present paper seem to be Morse, where we omit explicit calculations in most of our explicit cases. In general, it may be difficult to calculate derivatives of the functions to investigate whether the functions are Morse or not.  
We also omit precise exposition on Morse functions and fundamental arguments on {\it Morse}({\it-Bott}) functions. For this, see \cite{milnor1, milnor2} and for related singularity theory, see \cite{golubitskyguillemin}.
\begin{Thm}
\label{thm:2}
\begin{enumerate}
\item \label{thm:2.1} A point $p \in X_{m,c_1,c_2}$ is a critical point of ${\pi}_{2,1} \circ f_{c_1,c_2}={\pi}_{m+1,1} {\mid}_{X_{m,c_1,c_2}}$ if and only if $p$ is in $S({\pi}_{m+1,2} {\mid}_{X_{m,c_1,c_2}})=\{(x_1,x_2,{(0)}_{j=1}^{m-1}) \mid (x_1,x_2) \in \overline{D_{c_1,c_2}}-D_{c_1,c_2}\}$ and of the form $(c_i(x),x,{(0)}_{j=1}^{m-1})$ {\rm (}$x \in S(c_i)${\rm )}.
\item \label{thm:2.2} If $c_1$ and $c_2$ are Morse functions, then the function ${\pi}_{m+1,1} {\mid}_{X_{m,c_1,c_2}}$ is Morse.
\end{enumerate}

\end{Thm}

Main Theorem \ref{mthm:1} is a slightly extended version of Main Theorem \ref{mthm:0} (\ref{mthm:0.1}).
\begin{MainThm}
\label{mthm:1}
 For two $\mathbb{R}$-$\mathbb{R}$ TS functions $c_1$ and $c_2$ such that the set $c_1(S(c_1)) \bigcup c_2(S(c_2))$ is a discrete and closed set of $\mathbb{R}-Z_{\rm F}$ with a suitably chosen discrete and closed subset $Z_{\rm F}$ of $\mathbb{R}$ and with each $p \in Z_{\rm F}$ satisfying the following conditions, the Reeb space $R_{{\pi}_{2,1} \circ f_{c_1,c_2}}=R_{{\pi}_{m+1,1} {\mid}_{X_{m,c_1,c_2}}}$ is regarded to be a $1$-dimensional CW complex.
\begin{itemize}
\item For each $p \in Z_{\rm F}$ and each contour $C_p$ in each level set ${{\pi}_{2,1}}^{-1}(p) \bigcap \overline{D_{c_1,c_2}}$ of the function ${\pi}_{2,1} {\mid}_{\overline{D_{c_1,c_2}}}$, we have a small connected neighborhood $K(C_p)$ of $C_p$ in $\overline{D_{c_1,c_2}}$ represented as the disjoint union of some contours of ${\pi}_{2,1} {\mid}_{\overline{D_{c_1,c_2}}}$ {\rm (}a {\rm saturated neighborhood} in \cite{gelbukh1, saeki3}{\rm )}.
\item The set $K(C_p)$ is a closed subset of $\overline{D_{c_1,c_2}}$.
\item $K(C_p)-C_p$ does not contain any point of the form $(c_i(x),x)$ {\rm (}$x \in S(c_i)${\rm )}.
\end{itemize}
\end{MainThm}
\begin{proof}
This is essentially same as \cite[Theorem 2 (2)]{kitazawa9}. As the proof of the presented theorem, we use methods from \cite{saeki2, saeki3} and \cite{gelbukh1}, for example, although we do not need to understand such general cases. We may refer to the preprint \cite{kitazawa7}, arguments in which are also respected, and we do not assume related knowledge.

We use $\mathbb{R}$-$\mathbb{R}$ TS functions $c_1$ and $c_2$. From this, for any $p \in {\pi}_{2,1}(\overline{D_{c_1,c_2}})$, each contour $C_p$ in each level set ${{\pi}_{2,1}}^{-1}(p) \bigcap \overline{D_{c_1,c_2}}$ of the function ${\pi}_{2,1} {\mid}_{\overline{D_{c_1,c_2}}}$ is represented as a component of the intersection of the line $\{(p,t) \mid t \in \mathbb{R}\}$ and $\overline{D_{c_1,c_2}}$. This is homeomorphic to $D^1$, $\mathbb{R}$, or $\{t>0\} \subset \mathbb{R}$ and closed in $\overline{D_{c_1,c_2}}$.

For $p \in Z_{\rm F}$, we choose a pair of real numbers $p_{C_p,1}<p_{C_p,2}$ with $p_{C_p,1}<p<p_{C_p,2}$ and the difference $p_{C_p,2}-p_{C_p,1}$ being sufficiently small and choose a connected component $K_0(C_p)$ of $K(C_p) \bigcap \{(p,t) \mid p_{C_p,1} \leq p \leq p_{C_p,2}, t \in \mathbb{R}\}$ containing $C_p$. This is also closed in $\overline{D_{c_1,c_2}}$. By the assumption on the TS functions $c_1$ and $c_2$ and their critical sets, these closed and connected sets $K_0(C_p)$ in $\overline{D_{c_1,c_2}}$ can be also chosen in such a way that each $K_0(C_{p_1})$ chosen for each point $p_1 \in Z_{\rm F}$ does not contain any point of the form $(c_i(x),x,{(0)}_{j=1}^{m-1}) \in K_0(C_{p_1})-C_{p_1}$ {\rm (}$x \in S(c_i)${\rm )} (, as in the assumption,) or any connected component $C_{p_2}$ of this type chosen for a point $p_2 \in Z_{\rm F}$, as a subset other than $C_{p_1}$. This means that the two distinct contours $C_{p_1}$ and $C_{p_2}$ of the function ${\pi}_{m+1,1} {\mid}_{X_{m,c_1,c_2}}$ are separated by suitably chosen neighborhoods in $\overline{D_{c_1,c_2}}$. 

The set $c_1(S(c_1)) \bigcup c_2(S(c_2))$ is assumed to be discrete and closed in $\mathbb{R}-Z_{\rm F}$ for some discrete and closed subset $Z_{\rm F}$ of $\mathbb{R}$.
Due to this assumption, the condition on the existence of a similar connected neighborhood $K(C_{0,p})$ of  $C_{0,p}$ and another similar connected set $K_0(C_{0,p})$ which are also closed subsets of  $\overline{D_{c_1,c_2}}$ is also satisfied, for each point $p \in {\pi}_{2,1}(\overline{D_{c_1,c_2}}) \bigcap  (c_1(S(c_1)) \bigcup c_2(S(c_2)) \bigcup (({\mathbb{R}}-{Z_{\rm F}}) \bigcap (\mathbb{R}-(c_1(S(c_1)) \bigcup c_2(S(c_2))))) \bigcup Z_{\rm F})$ and every contour $C_{0,p}$ of every level set ${{\pi}_{2,1}}^{-1}(p) \bigcap \overline{D_{c_1,c_2}}$ of the function ${\pi}_{2,1} {\mid}_{\overline{D_{c_1,c_2}}}$. We can see that two contours $C_{p_1}$ and $C_{p_2}$ of the function ${\pi}_{m+1,1} {\mid}_{X_{m,c_1,c_2}}$ are always separated by suitably chosen neighborhoods in $\overline{D_{c_1,c_2}}$. This guarantees that the Reeb space $R_{{\pi}_{2,1} \circ f_{c_1,c_2}}=R_{{\pi}_{m+1,1} {\mid}_{X_{m,c_1,c_2}}}$ is a Hausdorff space.

For each connected component $C_{K_0(C_{0,p})}$ of $K_0(C_{0,p})-C_{0,p}$, take its closure in $K_0(C_{0,p})$.

 By the previous argument, $C_{K_0(C_{0,p})}$ does not contain any point of the form $(c_i(x),x,{(0)}_{j=1}^{m-1})$ {\rm (}$x \in S(c_i)${\rm )} and $K_0(C_{0,p})$ contains exactly one contour $C_{0,p}$ of ${\pi}_{2,1} {\mid}_{\overline{D_{c_1,c_2}}}$ in the level set ${{\pi}_{2,1}}^{-1}(p)$. 

By the representations of the manifolds, the maps and the sets, together with the uniqueness of a contour of  ${\pi}_{2,1} {\mid}_{\overline{D_{c_1,c_2}}}$ in the level set ${{\pi}_{2,1}}^{-1}(p) \bigcap \overline{D_{c_1,c_2}} ={{\pi}_{2,1}}^{-1}(p) \bigcap K_0(C_p)$, $C_{K_0(C_{0,p})}$ is represented as a connected component
 of a set of the form $\overline{D_{c_1,c_2}} \bigcap \{(p^{\prime},t) \mid p_{C_p,0,1} \leq p^{\prime} \leq p_{C_p,0,2}, t \in \mathbb{R}\}$ with some real numbers $p_{C_p,0,1}<p_{C_p,0,2}$ or either $\overline{D_{c_1,c_2}} \bigcap \{(p^{\prime},t) \mid p_{C_p,0} \leq p^{\prime}<p, t \in \mathbb{R}\}$ or $\overline{D_{c_1,c_2}} \bigcap \{(p^{\prime},t) \mid p \leq p^{\prime} \leq p_{C_p,0} , t \in \mathbb{R}\}$ with some real number $p_{C_p,0}$. By the non-existence of points of the form $(c_i(x),x,{(0)}_{j=1}^{m-1})$ {\rm (}$x \in S(c_i) \bigcap K_0(C_p)-C_p${\rm )}, $C_{K_0(C_{0,p})}$ is, by the quotient map $q_{{\pi}_{2,1} {\mid}_{\overline{D_{c_1,c_2}}}}$ mapped to the set homeomorphic to a $1$-dimensional smooth manifold of the form $\{p^{\prime} \mid p_{C_p,0,1} \leq p^{\prime} \leq p_{C_p,0,2}\}$, or either $\{p^{\prime} \mid p_{C_p,0} \leq p^{\prime}<p\}$ or $\{p^{\prime} \mid p<p^{\prime} \leq p_{C_p,0}\}$.
Suppose that the intersection $\overline{C_{K_0(C_{p})}}^{K_0(C_p) \bigcap C_p}$ of the closure $\overline{C_{K_0(C_{p})}}^{K_0(C_p)}$ of $C_{K_0(C_{p})}$  in $K_0(C_p)$ is empty and we prove that this yields a contradiction. By the representations of the manifolds, the maps and the sets, together with the uniqueness of a contour of  ${\pi}_{2,1} {\mid}_{\overline{D_{c_1,c_2}}}$ in the level set ${{\pi}_{2,1}}^{-1}(p) \bigcap \overline{D_{c_1,c_2}} ={{\pi}_{2,1}}^{-1}(p) \bigcap K_0(C_p)$, the closure $\overline{C_{K_0(C_{p})}}^{K_0(C_p)}$ must not have any point mapped to $p$ by ${\pi}_{2,1}$. 
We can see that $\overline{C_{K_0(C_{p})}}^{K_0(C_p)}$ must be a non-empty connected component of $K_0(C_p)$ and that $C_p$ is contained in its another connected component. This is a contradiction. 

The Reeb space $R_{{\pi}_{m+1,1} {\mid}_{X_{m,c_1,c_2}}}$  is Hausdorff and locally connected. By our situation, at each point $p_{C} \in R_{{\pi}_{m+1,1} {\mid}_{X_{m,c_1,c_2}}}$ the Reeb space is locally homeomorphic to $\mathbb{R}$ or a space obtained by identifying points $0 \in \mathbb{R}$ in finitely many or countably many disjoint $1$-dimensional manifolds represented in the form $\{x \geq 0 \mid x \in \mathbb{R}\}$. Note that in the latter case, $p_C$ is identified with the points $0$.

This completes the proof.
\end{proof}
\begin{Ex}
\label{ex:1}
\begin{enumerate}
\item \label{ex:1.1} Consider $c_1(x):=\sin x$ and $c_2(x):=c_1(x)+a$ with $a>0$. The case $a>2$ gives an important case for Main Theorem \ref{mthm:1} with $Z_{\rm F}$ being empty. 
The 1st derivative is calculated as ${c_1}^{\prime}(x)=\cos x$ and the 2nd derivative is calculated as ${c_1}^{\prime \prime}(x)=-\sin x$ for $c_1$.
This is also a case for Theorem \ref{thm:2} (\ref{thm:2.2}). 
\item \label{ex:1.2} \cite[Example 1(2)]{kitazawa9} is also an important case for Main Theorem \ref{mthm:1} with $Z_{\rm F}=\{0\}$.
\end{enumerate}
\end{Ex}
\begin{Rem}
\label{rem:1}
In Theorem \ref{thm:1}, $X_{m,c_1,c_2}$ is diffeomorphic to $\mathbb{R} \times S^{m-1}$. Let $c_1$ be a constant function whose values are always $-1$ and let $a=2$ in Example \ref{ex:1} (\ref{ex:1.1}). This case yields a most natural geometric realization of $\mathbb{R} \times S^{m-1} \subset {\mathbb{R}}^{m+1}$.   
\end{Rem}
\begin{MainThm}
\label{mthm:2}
Let $c_1:\mathbb{R} \rightarrow \mathbb{R}$ be an $\mathbb{R}$-$\mathbb{R}$ TS function with $S(c_1)$ being non-empty and suppose that the absolute value $|c_{1,j}-c_{1,j+1}|$ is always greater than a positive number $N_{{\rm P},c_1,c_2}$ {\rm (}we consider all three cases $S(c_1)={\sqcup}_{j=1}^{l} I_j$, $S(c_1)={\sqcup}_{j=1}^{\infty} I_j$, and $S(c_1)={\sqcup}_{j=-\infty}^{\infty} I_j$ {\rm )}. 
Let $c_2(x):=c_1(x)+a$ with $0<a<N_{{\rm P},c_1,c_2}$. Then the Reeb space $R_{{\pi}_{m+1,1} {\mid}_{X_{m,c_1,c_2}}}$ is the Reeb graph whose vertices are of degree $1$, $2$, or $3$. This is also the Reeb digraph and its vertex where the function $\bar{{\pi}_{m+1,1} {\mid}_{X_{m,c_1,c_2}}}$ has a local extremum must be of degree $1$. 
\end{MainThm}
\begin{proof}
We investigate each critical contour of ${\pi}_{m+1,1} {\mid}_{X_{m,c_1,c_2}}$. Based on Theorem \ref{thm:2}, we investigate
 $(c_{1,j}(x),x)$ ($x \in I_{1,j}$) and $(c_{2,j}(x),x)$ ($x \in I_{2,j}=I_{1,j}$) and neighborhoods in ${\mathbb{R}}^2$ and prove the theorem. \\
\ \\
Case 1 $c_{1,j}$ is a local extremum of $c_{1}$.  \\
$c_{1,j}$ and $c_{2,j}$ are local extrema of $c_{1}$ and $c_{2}$, respectively. By the local behaviors of $c_1$ and $c_2$ and $0<a<N_{{\rm P},c_1,c_2}$, we have a subspace of the shape "Y", seen as an edge whose closure is homeomorphic to $D^1$ with two edges incident to one of the vertex of it, in the Reeb space $R_{{\pi}_{m+1,1} {\mid}_{X_{m,c_1,c_2}}}$, locally. \\
\ \\
Case 2 $c_{1,j}$ is not a local extremum of $c_{1}$.  \\
By the local behaviors of $c_1$ and $c_2$ and $0<a<N_{{\rm P},c_1,c_2}$, we have a subspace seen as a line with a pair of vertices of degree $2$ (of a graph), in the Reeb space, locally. \\
\ \\
By the observation on these cases, the Reeb space $R_{{\pi}_{m+1,1} {\mid}_{X_{m,c_1,c_2}}}$ is the Reeb graph whose vertex is of degree $1$, $2$, or $3$. The statement on the degrees of vertices for the Reeb digraph is easily shown.

This completes the proof.
\end{proof}

\begin{Ex}
\label{ex:2}
Consider $c_1(x):=\sin x$ and $c_2(x):=\sin x+1$ for Main Theorem \ref{mthm:2}. Furthermore, the graph is isomorphic to the graph presented first in the preprint \cite[Theorem 1]{kitazawa8}. There, the graph is obtained in a different way. The case $c_1(x)=\sin x$ and $c_2(x)=\sin x+1$ is not presented in \cite{kitazawa7, kitazawa8, kitazawa9}.
\end{Ex}
\begin{Rem}
\label{rem:2}
In Main Theorem \ref{mthm:2}, if we consider the case with $S(c_1)$ being empty instead, then the Reeb space is homeomorphic to $\mathbb{R}$ for any $c_2=c_1+a$ ($a \in \mathbb{R}$).
\end{Rem}
Hereafter, $0 \in {\mathbb{R}}^k$ is used for the origin of ${\mathbb{R}}^k$
\begin{MainThm}
\label{mthm:3}
Let $c_{0}$ be an $ \mathbb{R}$-$\mathbb{R}$ TS function $c_{0}$ which is bounded and whose 1st derivative ${c_{0}}^{\prime}(x)$ satisfies $-\frac{1}{a_{c,{\rm M}}} \leq {c_{0}}^{\prime}(x) \leq a_{c,{\rm m}}$ for every $x \in \mathbb{R}$, for some two real numbers $a_{c,{\rm m}}, a_{c,{\rm M}}>0$ satisfying $a_{c,{\rm m}}<a_{c,{\rm M}}$. Let $0<a_c<a_{c,{\rm M}}$ and we also assume 
that the level set ${{c_{0}}^{\prime}}^{-1}(a_{c})$ is a non-empty set ${\sqcup}_{j=1}^l {I_{0,j}}$ with $l$ being a positive integer, ${\sqcup}_{j=1}^{\infty} {I_{0,j}}$, or ${\sqcup}_{j=-\infty}^{\infty} {I_{0,j}}$, with $I_{0,j}$ being a closed and connected non-empty set as in the 1st section, just before Main Theorem \ref{mthm:0}, and that the union of tangent lines of the graph $\{(c_0(x),x) \mid x \in \mathbb{R}\}$ at all points of the image ${c_0}({{c_{0}}^{\prime}}^{-1}(a_{c}))$ is the disjoint union of at most countably many tangent lines and also a closed set of ${\mathbb{R}}^2$. By a suitable rotation, we have another function $c_1$ and $c_2:=c_1+a$ {\rm (}$a>0${\rm )} with the following.
\begin{enumerate}
\item \label{mthm:3.1} The two functions $c_1$ and $c_2$ give a case for Main Theorem \ref{mthm:1} with $Z_{\rm F}$ being empty.
\item \label{mthm:3.2} The sets $S(c_1)$ and $S(c_2)$ are both non-empty sets.
\item \label{mthm:3.3} Both the values $c_1(x)$ and $c_2(x)$ diverge to $+\infty$ as $x$ diverges to $-\infty$.
\item \label{mthm:3.4} Both the values $c_1(x)$ and $c_2(x)$ diverge to $-\infty$ as $x$ diverges to $+\infty$.

\end{enumerate}
\end{MainThm}
\begin{proof}
By a suitable rotation of the graph $\{(c_0(x),x) \mid x \in \mathbb{R}\}$, around the origin $0 \in {\mathbb{R}}^2$, the tangent lines of the graph $\{(c_0(x),x) \mid x \in \mathbb{R}\}$ at all points of the image ${c_0}({{c_{0}}^{\prime}}^{-1}(a_{c,{\rm m}}))$ are transformed to lines parallel to the line $\{(0,x_2) \mid x_2 \in \mathbb{R}\}$.

 The union of the tangent lines of the graph $\{(c_0(x),x) \mid x \in \mathbb{R}\}$ at all points of the image ${c_0}({{c_{0}}^{\prime}}^{-1}(a_{c}))$ is the disjoint union of at most countably many tangent lines and also a closed set of ${\mathbb{R}}^2$. In addition, we have $-1<(-\frac{1}{a_{c,{\rm M}}}) \times  a_{c,{\rm m}}<0$. This means that two functions $c_1$ and $c_2$ for Main Theorem \ref{mthm:1} can be naturally defined, where $Z_{\rm F}$ is empty. We can see that the sets $S(c_1)$ and $S(c_2)$ are both non-empty sets. We have shown that the properties (\ref{mthm:3.1}, \ref{mthm:3.2}) are enjoyed.  We investigate asymptotic behaviors of $c_1$ and $c_2$ to show that the properties (\ref{mthm:3.3}, \ref{mthm:3.4}) are also enjoyed and to complete the proof. The function $c_0$ is bounded. This means that the relation $t_{{\rm m},c_0}<c_0(x)<t_{{\rm M},c_0}$ holds for every $x \in \mathbb{R}$, for some real numbers $t_{{\rm m},c_0}<0$ and $t_{{\rm M},c_0}>0$. This implies that the remaining properties on asymptotic behaviors are also enjoyed. This completes the proof.   
\end{proof}

\begin{Ex}
\label{ex:3}
\begin{enumerate}
\item \label{ex:3.1} We consider the case $c_0(x):=\frac{1}{x^2+1}$. The 1st derivative is calculated as ${c_0}^{\prime}(x)=\frac{2x}{{(x^2+1)}^2}$. The functions $c_0$ and ${c_0}^{\prime}$ are bounded and converge to $0$ at each infinity. In this case, we can choose $a_{c,{\rm m}}$ as the supremum of ${c_0}^{\prime}$ and an arbitrary number $a_{c,{\rm M}}>a_{c,{\rm m}}$. For each $0<a_c \leq a_{c,{\rm m}}$, we can have a case of Main Theorem \ref{mthm:3} with the sets $S(c_1)$ and $S(c_2)$ being bounded and finite. This is also for \cite[Theorem 4]{kitazawa9}.
\item \label{ex:3.2} We consider the case $c_0(x):=\sin x$. In this case, we can choose $a_{c,{\rm m}}=1$ and $a_{c,{\rm M}}=\frac{3}{2}$. For each $0<a_c \leq a_{c,{\rm m}}=1$, we can have a case of Main Theorem \ref{mthm:3} with the sets $S(c_1)$ and $S(c_2)$ being closed, discrete, and unbounded. This is also for \cite[Theorem 4]{kitazawa9}. This is also for Theorem \ref{thm:2} (\ref{thm:2.1}), by remembering exposition on derivatives in Example \ref{ex:1}.
\end{enumerate}
\end{Ex}

Hereafter, two subsets $P_1$ and $P_2$ in ${\mathbb{R}}^k$ are {\it congruent} if there exists a map mapping $P_1$ onto $P_2$ and represented by a composition of finitely many maps each of which is a rotation, a reflection, or a parallel transformation. This yields an equivalence relation on the set of all subsets in ${\mathbb{R}}^k$. Two graphs $\{(c_{0,1}(x),x) \mid  x \in \mathbb{R}\}$ and $\{(c_{0,2}(x),x) \mid \mathbb{R}\}$ of functions $c_{0,i}:\mathbb{R} \rightarrow \mathbb{R}$ (which may be not smooth or continuous) are {\it globally similar} if these sets are congruent or deformed to two graphs $\{(c_{1,1}(x),x)\}$ and $\{(c_{1,2}(x),x)\}$ of functions $c_{1,i}:\mathbb{R} \rightarrow \mathbb{R}$ by some transformations for obtaining new congruent subsets respectively in such a way that we have $c_{1,2}(x):=c_{1,1}(x)+p(x)$ with a smooth bounded function $p:\mathbb{R} \rightarrow \mathbb{R}$. Of course this relation is an equivalence relation on the set of all real-valued functions on $\mathbb{R}$. 

We also use phrases such as "{\it up to congruence}" and "{\it up to global similarity}", for example.

We can check the following easily. It is essential that for a bounded function $c:\mathbb{R} \rightarrow \mathbb{R}$, $t_{{\rm m},c}<c(x)<t_{{\rm M},c}$ holds for every $x \in \mathbb{R}$, for some real numbers $t_{{\rm m},c}<0$ and $t_{{\rm M},c}>0$ and this fundamental fact is presented in the proof of Main Theorem \ref{mthm:3}.
\begin{Cor}
\label{cor:1}
For a bounded function $c_1:\mathbb{R} \rightarrow \mathbb{R}$ and another function $c_2:\mathbb{R} \rightarrow \mathbb{R}$ such that the graphs $\{(c_1(x),x) \mid x \in \mathbb{R}\}$ and $\{(c_2(x),x) \mid x \in \mathbb{R}\}$ are congruent, we can choose a suitable real number $T_{c_1,c_2}$ and a constant real valued function $p_{c_1.c_2}$ and have the relation $c_2(x)=p_{c_1,c_2}+c_1(x-T_{c_1,c_2})$ or $c_2(x)=p_{c_1,c_2}-c_1(x-T_{c_1,c_2})$.
\end{Cor}

\begin{Def}
\label{def:1}
For an equivalence class on the set of all graphs $\{(c(x),x) \mid x \in \mathbb{R}\}$ of all real valued functions $c:\mathbb{R} \rightarrow \mathbb{R}$ up to congruence, if the property as in Corollary \ref{cor:1} is enjoyed, then the class is said to be {\it rigid with respect to rotations} or {\it Rot-rigid}.
\end{Def}

We give answers to \cite[Problems 2 and 3]{kitazawa9}, in Main Theorems \ref{mthm:4} and \ref{mthm:5}.

\begin{Prop}
\label{prop:1}
Let $p:\mathbb{R} \rightarrow \mathbb{R}$ be a function which is always positive and is represented as the sum of finitely many functions of the form $p_j(x):=a_jx^{r_{a_j}}$ with $a_j \in \mathbb{R}-\{0\}$ and integers $r_{a_j}>1$. We can choose i.e. $p(x)=x^2+1$ for example.

A smooth function $c_{p,0}:\mathbb{R} \rightarrow \mathbb{R}$ defined by $c_{p,0}(x):=\frac{2+\sin (e^{x^2})}{p(x)}$ enjoys the following properties.
\begin{enumerate}
\item The values of the function $c_{p,0}$ is always positive. The value $c_{p,0}(x)$ converges to $0$ as $x$ diverges to $\pm \infty$.
\item The 1st derivative ${c_{p,0}}^{\prime}(x)$ of $c_{p,0}(x)$ is calculated as\\ ${c_{p,0}}^{\prime}(x)=\frac{2xe^{x^2} \times \cos(e^{x^2}) \times {p(x)}-(2+\sin (e^{x^2}))\times p^{\prime}(x)}{{p(x)}^2}$.

\end{enumerate}
\end{Prop}This is a kind of fundamental exercises from elementary calculus.
By considering $x$ satisfying $\cos (e^{x^2})=\pm 1$ and asymptotic behaviors of $e^{x^2}$ and real polynomial functions (or so-called {\it real algebraic} functions), we have the following.
\begin{Prop}
\label{prop:2}
Let $c_{p,0}$ be a function of Proposition \ref{prop:1}.
There exist a sequence $\{r_{{\rm P},j}\}_{j=1}^{\infty}$ of positive numbers diverging to $\infty$
such that the sequence $\{{c_{p,0}}^{\prime}(r_{{\rm P},j})\}_{j=1}^{\infty}$ diverges to $\infty$
 and one such that the sequence $\{{c_{p,0}}^{\prime}(r_{{\rm P},j})\}_{j=1}^{\infty}$ diverges to $-\infty$.
There exist a sequence $\{r_{{\rm N},j}\}_{j=1}^{\infty}$ of negative numbers diverging to $-\infty$
such that the sequence $\{{c_{p,0}}^{\prime}(r_{{\rm N},j})\}_{j=1}^{\infty}$ diverges to $\infty$
 and one such that the sequence $\{{c_{p,0}}^{\prime}(r_{{\rm N},j})\}_{j=1}^{\infty}$ diverges to $-\infty$.

\end{Prop}
We can have a connected component of a hyperbola, an important real algebraic curve of degree $2$ in ${\mathbb{R}}^2$. For any $r>0$, we can define a smooth function $c_{{\rm H},+\infty,r}:\mathbb{R} \rightarrow \mathbb{R}$ such that the graph is a connected component of a hyperbola congruent to a hyperbola of the form $\{(x_1,x_2) \in {\mathbb{R}}^2 \mid {x_1}^2-r{x_2}^2=1\}$ and that the value $c_{{\rm H},+\infty,r}(t)$ diverges to $+\infty$ as $t$ diverges to $\pm \infty$. We can have the function with the following properties, uniquely.
\begin{itemize}

\item The critical set $S(c_{{\rm H},+\infty,r})$ of $c_{{\rm H},+\infty,r}$ consists of exactly one point, which is $0 \in \mathbb{R}$.
\item The minimum of the function is $1$. $c_{{\rm H},+\infty,r}(x)=c_{{\rm H},+\infty,r}(-x)$.
\item The 1st derivative ${c_{{\rm H},+\infty,r}}^{\prime}$ of $c_{{\rm H},+\infty,r}$ is bounded.
\item The critical set $S({c_{{\rm H},+\infty,r}}^{\prime})$ of the 1st derivative ${c_{{\rm H},+\infty,r}}^{\prime}$ is empty.
\end{itemize}
Main Theorem \ref{mthm:4} gives a case similar to a case of Main Theorem \ref{mthm:3}. This is also an example for \cite[Theorem 4]{kitazawa9} with $S(c_1) \bigcup S(c_2) \subset \mathbb{R}$ being unbounded and both the values $c_1(x)$ and $c_2(x)$ diverging to $+\infty$ as $x$ diverges to $\pm \infty$. Such a case is not in the original preprint \cite{kitazawa9}. 

\begin{MainThm}
\label{mthm:4}
By considering $c_1:=c_{{\rm H},+\infty,r}$ above and $c_2:=c_{{\rm H},+\infty,r}+c_{p,0}$ for some $c_{p,0}$ which is from Proposition \ref{prop:1} and the real function $p:\mathbb{R} \rightarrow \mathbb{R}$ is always positive such as $p(x)=x^2+1$, we have a case with the following properties.
\begin{enumerate}
\item \label{mthm:4.1} The two functions $c_1$ and $c_2$ give a case for Main Theorem \ref{mthm:1} with $Z_{\rm F}$ being empty.
\item \label{mthm:4.2} The set $S(c_1) \bigcup S(c_2) \subset \mathbb{R}$ is unbounded.
\item \label{mthm:4.3} Both the values $c_1(x)$ and $c_2(x)$ diverge to $+\infty$ as $x$ diverges to $-\infty$.
\item \label{mthm:4.4} Both the values $c_1(x)$ and $c_2(x)$ diverge to $+\infty$ as $x$ diverges to $+\infty$.

\end{enumerate}
\end{MainThm}
We can prove easily from the arguments on functions above. Note that $\{(c_1(x),x) \mid x \in \mathbb{R}\}$ and $\{(c_2(x),x) \mid x \in \mathbb{R}\}$ are globally similar.

Note that $c_1:=c_{{\rm H},+\infty,r}$ is not R-rigid up to congruence. By rotating the graph $\{(c_{{\rm H},+\infty,r}(x),x) \mid x \in \mathbb{R}\}$ slightly, we can understand this.

\begin{Prop}
\label{prop:3}
Let $p:\mathbb{R} \rightarrow \mathbb{R}$ be a function which is always positive and is represented as the sum of finitely many functions of the form $p_j(x):=a_jx^{r_{a_j}}$ with $a_j \in \mathbb{R}-\{0\}$ and integers $r_{a_j}>1$. We can choose i.e. $p(x)=x^2+1$ for example.

A smooth function $c_{p,0,0}:\mathbb{R} \rightarrow \mathbb{R}$ defined by $c_{p,0,0}(x):=\frac{2+\sin (e^{x})}{p(x)}$ enjoys the following properties.
\begin{enumerate}
\item The values of the function $c_{p,0,0}$ is always positive. The value $c_{p,0,0}(x)$ converges to $0$ as $x$ diverges to $\pm \infty$.
\item The 1st derivative ${c_{p,0,0}}^{\prime}(x)$ of $c_{p,0,0}(x)$ is calculated as\\ ${c_{p,0,0}}^{\prime}(x)=\frac{e^{x} \times \cos(e^{x}) \times {p(x)}-(2+\sin (e^{x}))\times p^{\prime}(x)}{{p(x)}^2}$.

\end{enumerate}
\end{Prop}

This is also a kind of fundamental exercises from elementary calculus.
By considering $x$ satisfying $\cos (e^{x})=\pm 1$ in $x>0$ and asymptotic behaviors of $e^{x}$ and real polynomial functions (or so-called {\it real algebraic} functions), we have the following.
\begin{Prop}
\label{prop:4}
Let $c_{p,0,0}$ be a function of Proposition \ref{prop:3}.
There exist a sequence $\{r_{{\rm P},j}\}_{j=1}^{\infty}$ of positive numbers diverging to $\infty$
such that the sequence $\{{c_{p,0,0}}^{\prime}(r_{{\rm P},j})\}_{j=1}^{\infty}$ diverges to $\infty$
 and one such that the sequence $\{{c_{p,0,0}}^{\prime}(r_{{\rm P},j})\}_{j=1}^{\infty}$ diverges to $-\infty$.
The value ${c_{p,0,0}}^{\prime}(x)$ converges to $0$ as $x$ diverges to $-\infty$.

\end{Prop}
Main Theorem \ref{mthm:5} is an answer to \cite[Problem 2 and Problem 3]{kitazawa9}. 

\begin{MainThm}
\label{mthm:5}
Let $p_1 \in \mathbb{R}$. By choosing two globally similar real analytic functions $c_1$ and $c_2$ suitably, we have a case with the following properties.
\begin{enumerate}
\item \label{mthm:5.1} The two functions $c_1$ and $c_2$ give a case for Main Theorem \ref{mthm:1} with $Z_{\rm F}=\{p_1\}$. Their images are both $\{x>p_1 \mid x \in \mathbb{R}\}$.
\item \label{mthm:5.2} The set $S(c_1)$ is empty and the set $S(c_2) \subset \mathbb{R}$ is unbounded.
\item \label{mthm:5.3} Both the values $c_1(x)$ and $c_2(x)$ converge to $p_1$ as $x$ diverges to $-\infty$.
\item \label{mthm:5.4} Both the values $c_1(x)$ and $c_2(x)$ diverge to $+\infty$ as $x$ diverges to $+\infty$.
\end{enumerate}
We can also have a case with the properties {\rm (}\ref{mthm:5.1}, \ref{mthm:5.3}, \ref{mthm:5.4}, \ref{mthm:5.5}{\rm )}.
\begin{enumerate}
\setcounter{enumi}{4}
\item \label{mthm:5.5} The sets $S(c_1) \subset \mathbb{R}$ and $S(c_2) \subset \mathbb{R}$ are unbounded.
\end{enumerate}

We can also have a case with the properties {\rm (}\ref{mthm:5.1}, \ref{mthm:5.3}, \ref{mthm:5.4}, \ref{mthm:5.6}{\rm )}.
\begin{enumerate}
\setcounter{enumi}{5}
\item \label{mthm:5.6} The sets $S(c_1) \subset \mathbb{R}$ and $S(c_2) \subset \mathbb{R}$ are bounded below and unbounded above.
\end{enumerate}
\end{MainThm}
\begin{proof}
For $p_1 \in \mathbb{R}$ and $p_2>0$, we can define a smooth function $c_{{\rm H},p_1,p_2}:\mathbb{R} \rightarrow \mathbb{R}$ such that the graph is a connected component of a hyperbola congruent to a hyperbola of the form $\{(x_1,x_2) \in {\mathbb{R}}^2 \mid x_1x_2+p_2=0\}$ and that the value $c_{{\rm H},r,p_1,p_2}(t)$ converges to $p_1$ as $t$ diverges to $-\infty$ and that the value $c_{{\rm H},p_1,p_2}(t)$ diverges to $+\infty$ as $t$ diverges to $+\infty$. In addition, we can obtain the function with the following properties. 
\begin{itemize}
\item The critical set $S(c_{{\rm H},p_1,p_2})$ of $c_{{\rm H},p_1,p_2}$ is empty.
\item The 1st derivative ${c_{{\rm H},p_1,p_2}}^{\prime}$ is bounded.
\item The critical set $S({c_{{\rm H},p_1,p_2}}^{\prime})$ of the 1st derivative ${c_{{\rm H},p_1,p_2}}^{\prime}$ is empty.

\end{itemize}
We remember Proposition \ref{prop:2}.

We also remember the equation for a hyperbola of the form $x_1x_2+p_2=0$ ($x_2=-\frac{p_2}{x_1}$). We choose a certain transformation of coordinates via linear transformation to $x_{1,1}=x_1-\frac{p_3}{x_{1}}$ with a suitable number $p_3 \neq 0$ and $x_{2,1}=x_2$, composed with a parallel transformation to $x_{1,2}=x_{1,1}$ and $x_{2,2}=x_{2,1}+p_1$, to have the graph $\{(x_{1,2},c_{{\rm H},p_1,p_2}(x_{1,2})) \mid x_{1,2} \in \mathbb{R}\}$. We also need to consider the asymptotic behavior at each infinity to make $c_1(x)>p_1$ hold for some ${p}_{\rm m}$ and each real number $x<{p}_{\rm m}$.

We can define $c_1(x_{1,2}):=c_{{\rm H},p_1,p_2}(x_{1,2})$ and $c_2(x_{1,2}):=c_{{\rm H},p_1,p_2}(x_{1,2})+c_{p,0}(x_{1,2})$, with $c_{p,0}$ from Propositions \ref{prop:1} and \ref{prop:2}, and have a case for the first statement. More precisely, we choose $p(x)=x^2+1$ for example.

We can define $c_1(x_{1,2}):=c_{{\rm H},p_1,p_2}(x_{1,2})+\frac{1}{2} \times c_{p,0}(x_{1,2})$ and $c_2(x_{1,2}):=c_{{\rm H},p_1,p_2}(x_{1,2})+c_{p,0}(x_{1,2})$, with $c_{p,0}$ from Propositions \ref{prop:1} and \ref{prop:2}, and have a case for the second statement. 

We define $c_1(x_{1,2}):=c_{{\rm H},p_1,p_2}(x_{1,2})+\frac{1}{2} \times c_{p,0,0}(x_{1,2})$ and $c_2(x_{1,2}):=c_{{\rm H},p_1,p_2}(x_{1,2})+c_{p,0,0}(x_{1,2})$, with $c_{p,0,0}$ from Propositions \ref{prop:3} and \ref{prop:4}. Remember the asymptotic behavior of $c_{p,0,0}$ and the 1st derivative ${c_{p,0,0}}^{\prime}$ and we can see that the addition of $\frac{1}{2} \times c_{p,0,0}(x_{1,2})$ and that of $c_{p,0,0}(x_{1,2})$ to $c_{{\rm H},p_1,p_2}(x_{1,2})$ do not change the signs of the vaues of the 1st derivatives at the infinity $-\infty$. We can see that this gives a case for the third statement. 

This completes the proof.
\end{proof}

Main Theorem \ref{mthm:6} is another answer to \cite[Problem 3]{kitazawa9} and we have a case of \cite[Proposition 3 and Theorem 3]{kitazawa9}. 

\begin{MainThm}
\label{mthm:6}
Let $p_1, p_2 \in \mathbb{R}$ be distinct two numbers. By choosing two globally similar real analytic functions $c_1$ and $c_2$ suitably, we have a case with the following properties.
\begin{enumerate}
\item \label{mthm:6.1} The two functions $c_1$ and $c_2$ give a case for Main Theorem \ref{mthm:1} with $Z_{\rm F}=\{p_1,p_2\}$. Their images are both $\{p_1<x<p_2 \mid x \in \mathbb{R}\}$ in the case $p_1<p_2$ and both $\{p_2<x<p_1 \mid x \in \mathbb{R}\}$ in the case $p_1>p_2$.
\item \label{mthm:6.2} The set $S(c_1)$ is empty and the set $S(c_2) \subset \mathbb{R}$ is unbounded.
\item \label{mthm:6.3} Both the values $c_1(x)$ and $c_2(x)$ converge to $p_1$ as $x$ diverges to $-\infty$.
\item \label{mthm:6.4} Both the values $c_1(x)$ and $c_2(x)$ converge to $p_2$ as $x$ diverges to $+\infty$.
\end{enumerate}
We can also have a case with the properties {\rm (}\ref{mthm:6.1}, \ref{mthm:6.3}, \ref{mthm:6.4}, \ref{mthm:6.5}{\rm )}.
\begin{enumerate}
\setcounter{enumi}{4}
\item \label{mthm:6.5} The sets $S(c_1) \subset \mathbb{R}$ and $S(c_2) \subset \mathbb{R}$ are unbounded.
\end{enumerate}
We can also have a case with the properties {\rm (}\ref{mthm:6.1}, \ref{mthm:6.3}, \ref{mthm:6.4}, \ref{mthm:6.6}{\rm )}.
\begin{enumerate}
\setcounter{enumi}{5}
\item \label{mthm:6.6} The sets $S(c_1) \subset \mathbb{R}$ and $S(c_2) \subset \mathbb{R}$ are bounded below and unbounded above.
\end{enumerate}
\end{MainThm}
\begin{proof}
We define a smooth function $c_{e,p_1,p_2}:\mathbb{R} \rightarrow \mathbb{R}$ by $c_{e,p_1,p_2}(x):=\frac{p_1e^x+p_2}{e^x+1}$. 
Its 1st derivative is calculated as $c^{\prime}(x):=\frac{{(p_1-p_2)}e^x}{{({e^x}+1)}^2}$.
For this, the following are enjoyed.

\begin{itemize}
\item The values $c_{e,p_1,p_2}(t)$ converge to $p_1$ ($p_2$) as $t$ diverges to $-\infty$ (resp. $\infty$).
\item $c_{e,p_1,p_2}(x)-p_1=\frac{p_2-p_1}{e^x+1}$ and $c_{e,p_1,p_2}(x)-p_2=\frac{p_1-p_2}{e^x+1}$.
\item The critical set $c_{e,p_1,p_2}$ of $c_{e,p_1,p_2}$ is empty.
\item The 1st derivative ${c_{e,p_1,p_2}}^{\prime}$ is bounded. The values ${c_{e,p_1,p_2}}^{\prime}$ are always positive in the case $p_1>p_2$ and always negative in the case $p_1<p_2$. The values of the function converge to $0$ at each infinity.

\end{itemize}

We remember Propositions \ref{prop:1} and \ref{prop:2}. We define a function of a similar type we can use in the present situation. In short, we define $c_{e,1}(x):=\frac{2+\sin (e^{x^4})}{e^{x^2}}$ instead.
The values of this function are always positive and the value converges to $0$ at each infinity. Its 1st derivative is calculated as ${c_{e,1}}^{\prime}(x)=\frac{(4x^3e^{x^4}\cos (e^{x^4}))(e^{x^2})-2x(2+\sin(e^{x^4}))e^{x^2}}{e^{2x^2}}$ and we have Proposition \ref{prop:2} for this function.

We remember Propositions \ref{prop:3} and \ref{prop:4}. We define a function of a similar type we can use in the present situation. In short, we define $c_{e,2}(x):=\frac{2+\sin (e^{x^3})}{e^{x^2}}$ instead.
The values of this function are always positive and the value converges to $0$ at each infinity. Its 1st derivative is calculated as ${c_{e,2}}^{\prime}(x)=\frac{(3x^2e^{x^3}\cos (e^{x^3}))(e^{x^2})-2x(2+\sin(e^{x^3}))e^{x^2}}{e^{2x^2}}$ and we have Proposition \ref{prop:4} for this function. Remember the asymptotic behavior of the 1st derivative ${c_{e,2}}^{\prime}(x)$ again.

The argument essentially same as that in the proof of Main Theorem \ref{mthm:5} completes the proof.

\end{proof}
\subsection{Future problems.}
We present future problems. These are also interdisciplinary.

Problem \ref{prob:3} is related to \cite{katokishimototsutaya2}. In their preprint, Kato, Kishimoto, and Tsutaya, have found Morse inequalities for Morse functions on non-compact Riemannian manifolds with no boundary. More precisely, so-called {\it amenable groups} such as the group $\mathbb{Z} \subset \mathbb{R}$ of all integers act on the manifolds preserving the (Riemannian) metrics and they consider bounded functions whose 1st derivatives and some higher derivatives are also bounded and which are "uniform" with respect to distributions and locations of their critical points. They also respect symmetries related to the actions of the groups, in addition. According to them, it is an essential assumption that the 1st derivatives of the bounded Morse functions are also bounded (this is also related to the paper \cite{katokishimototsutaya1}). It is also stated in \cite{katokishimototsutaya2} that the assumptions on higher derivatives may be only a kind of technical assumptions.

\begin{Prob}
\label{prob:3}
Can theory and conjectures on Morse functions from \cite{katokishimototsutaya2} be applied to our Morse functions, presented in Theorem \ref{thm:2} (\ref{thm:2.2}). Can we obtain Morse functions for (Main) Theorems of the present paper, explicitly? This may be difficult to calculate derivatives in general. For the structures of Riemannian manifolds on our manifolds, see Remark \ref{rem:1}, for example.
\end{Prob}

Problem \ref{prob:4} is on singularity theory of differentiable maps and so-called stability of these maps. Roughly, stability of smooth maps means that singularities of smooth maps are invariant under slight perturbations. This is a kind of classical theory in singularity theory. For this, see \cite{mather1, mather2, mather3, mather4, mather5, mather6, mather7, mather8}, papers on celebrated and sophisticated theory of Mather, and for a textbook on singularity theory of differentiable maps, see \cite{golubitskyguillemin}, for example. For the space of smooth maps from a closed manifold into another smooth manifold with no boundary, so-called {\it stable} maps exist densely if the pair of the dimensions of the manifolds is nice in the sense of the presented theory of Mather. For example, if the dimensions are sufficiently low, then the pair is nice in this sense. It has been difficult to find nice criteria to verify stability for non-proper maps, where fundamental criteria have been shown for proper maps. For Morse functions on closed manifolds, 
a smooth real-valued function on a smooth manifold with no boundary is stable if and only if it is a Morse function whose critical values are distinct at distinct critical values (of the function). Note that there are other kinds of stability of smooth maps, where in the case of proper maps, these are equivalent notions. In \cite{dimca}, Dimca has studied stability of non-proper (Morse) functions on $\mathbb{R}$. In \cite{hayano}, Hayano has studied various kinds of stability of non-proper (Morse) functions on general smooth manifolds with no boundary. For example, in \cite{hayano}, a sufficient condition for a Morse function to be stable is given and stability for some explicit Morse functions are also studied. For recent studies on stability, see also \cite{ichiki1, ichiki2}. According to these studies by Ichiki, in non-proper cases, stable maps do not exist densely.
\begin{Prob}
\label{prob:4}
Can we investigate stability of Morse functions in Theorem \ref{thm:2} (\ref{thm:2.2}) and Morse functions for some (Main) Theorems explicitly?
\end{Prob}

For this, we present an explicit case as Example \ref{ex:4}. We omit rigirous exposition of the notion of {\it stable} functions, {\it strongly} stable functions, and {\it infinitesimally} stable ones.   
For rigorous understanding, refer to the original paper \cite{hayano}, with the original papers of Mather, and the textbook \cite{golubitskyguillemin}. We only present partial or complete characterizations, or necessary, sufficient, or necessary and sufficient conditions for Morse functions to be of these classes.
\begin{Ex}
\label{ex:4}
Let $c_1(x):=e^{-x^2}\sin x$ and $c_2:=\frac{a_1}{x^2+1}+a_2$ with $a_1>0$ and $a_2 \geq 0$. 
$c_1$ is, according to \cite[Theorem 4.1]{hayano}, a stable and strongly stable Morse function which is not infinitesimally stable.
$c_2$ is a Morse function with exactly one critical point $0 \in \mathbb{R}$. 
This gives a case for Main Theorem \ref{mthm:1} with $Z_{\rm F}=\{0\} \notin c_1(S(c_1)) \bigcup c_2(S(c_2))$. This is regarded to be a case of \cite[Theorem 2(2)]{kitazawa9} with "$q_1=q_2=0$".
According to \cite[Theorem 1.1 and Theorem 4.1]{hayano}, if $a_1>0$ is sufficiently large and $a_2=0$, then the function ${\pi}_{m+1,1} {\mid}_{X_{m,c_1,c_2}}$ is stable and strongly stable. 
More precisely, according to \cite[Theorem 1.1 (1)]{hayano}, 
a Morse function  $c:X \rightarrow \mathbb{R}$ on a smooth manifold $X$ with no boundary is stable if the following two hold.
\begin{itemize}
	\item At distinct critical points of $c$ the values are always distinct. Equivalently, the restriction $c {\mid}_{S(c)}$ is injective.
	\item For each point $p \supset c(S(c))$, there exist a small open neighborhood $N_p$ in $\mathbb{R}$ and an open set $V_p$ of $X$ and the following hold.
\begin{itemize}
\item $X-V_p$ is compact.
\item $c^{-1}(p) \bigcap V_p \bigcap S(c)$ is empty.
\item There exists a diffeomorphism ${\Phi}_p:(c^{-1}(p) \bigcap V_p) \times N_p \rightarrow c^{-1}(N_p) \bigcap V_p$ and the function $c {\mid}_{c^{-1}(N_p) \bigcap V_p} \circ {\Phi}_p$ gives a smooth trivial bundle over $N_p$. 
\end{itemize}
\end{itemize}
In addition, according to \cite[Theorem 1.1 (2)]{hayano}, a Morse function $c:X \rightarrow \mathbb{R}$ on a smooth manifold $X$ with no boundary is strongly stable if and only if the following two hold.
\begin{itemize}
	\item The restriction $c {\mid}_{S(c)}$ is injective.
	\item For a small open neighborhood $N(c(S(c))) \supset c(S(c))$ in $\mathbb{R}$, the restriction of $c$ to the preimage $c^{-1}(N(c(S(c))))$ is proper.
	\end{itemize}
In our case, we can check this easily. 
In addition, in the case $a_2>0$ instead, the function is not strongly stable, where it is stable. In any case, the function is not infinitesimally stable, according to general theory of stability and exposition in \cite[Subsection 2.2]{hayano}. More precisely, the restriction ${\pi}_{m+1,1} {\mid}_{S({\pi}_{m+1,1} {\mid}_{X_{m,c_1,c_2}})}$ of our Morse function ${\pi}_{m+1,1} {\mid}_{X_{m,c_1,c_2}}$ to its critical set is non-proper around $0 \in \mathbb{R}$ in the target space $\mathbb{R}$: it is known as a necessary and sufficient condition for a Morse function $c$ such that $c {\mid}_{S(c)}$ is injective to be infinitesimally stable that $c {\mid}_{S(c)}$ is proper.  

Last, we can also see that the Reeb digraph $R_{{\pi}_{m+1,1} {\mid}_{X_{m,c_1,c_2}}}$ of this Morse function is a digraph whose vertices enjoy the property as in Main Theorem \ref{mthm:2}.
\end{Ex}

Problem \ref{prob:5} seems to be a kind of problems closest to our original motivation on studies on topologies and combinatorics of Reeb graphs and Reeb spaces of smooth functions. This comes from \cite{gelbukh1, gelbukh2, saeki2, saeki3}. For proper smooth real-valued functions and continuous real-valued functions on (nice) compact Hausdorff spaces, the Reeb spaces have nice topological properties. For example, for continuous functions on {\it Peano continuums}, defined as separable, compact, connected and locally-connected Hausdorff spaces, their Reeb spaces are also homeomorphic to such spaces.
\begin{Prob}
\label{prob:5}
What can we say about topological properties and combinatorial ones of Reeb spaces of non-proper smooth functions in general? Some presented (Main) Theorems are regarded as explicit answers in very specific situations. 
\end{Prob}
 \section{Conflict of interest and Data availability.}
  \noindent {\bf Conflict of interest.} \\
 The author is a researcher at Osaka Central Advanced Mathematical Institute (OCAMI researcher). This is supported by MEXT Promotion of Distinctive Joint Research Center Program JPMXP0723833165. He thanks this, where he is not employed there. \\
  \ \\
  {\bf Data availability.} \\
  No data other than the present article is generated. We do not assume non-trivial arguments in preprints which are still unpublished, where we may refer to these preprints.

\end{document}